\documentclass[a4paper,11pt]{amsart}

\usepackage[utf8]{inputenc}
\usepackage[T1]{fontenc}
\usepackage{lmodern}
\usepackage[english]{babel}
\usepackage[margin=32mm,bottom=40mm]{geometry}
\usepackage[hidelinks]{hyperref}
\usepackage{amsmath,amsthm,amssymb}
\usepackage{mathtools}
\usepackage[all]{xy}
\usepackage{microtype}

\newtheorem{theorem}{Theorem}
\newtheorem{proposition}[theorem]{Proposition}
\newtheorem{lemma}[theorem]{Lemma}
\newtheorem{corollary}[theorem]{Corollary}
\theoremstyle{definition}

\newtheorem{remark}[theorem]{Remark}
\theoremstyle{remark}
\newtheorem*{claim}{Claim}

\DeclareMathOperator{\Pic}{Pic}
\DeclareMathOperator{\Sym}{Sym}
\DeclareMathOperator{\Hom}{Hom}
\DeclareMathOperator{\Cl}{Cl}
\let\Im\relax
\DeclareMathOperator{\Im}{Im}
\newcommand*{\van}{\ensuremath{\mathrm{van}}}
\newcommand*{\ZZ}{\ensuremath{\mathbb{Z}}}

\newcommand*{\RR}{\ensuremath{\mathbb{R}}}
\newcommand*{\CC}{\ensuremath{\mathbb{C}}}
\newcommand*{\PP}{\ensuremath{\mathbb{P}}}
\newcommand*{\id}{\ensuremath{\mathrm{id}}}
\newcommand*{\Xcal}{\ensuremath{\mathcal{X}}}
\newcommand*{\Ecal}{\ensuremath{\mathcal{E}}}
\newcommand*{\Ocal}{\ensuremath{\mathcal{O}}}
\newcommand*{\Hcal}{\ensuremath{\mathcal{H}}}
\let\ol\overline
\let\eps\varepsilon

\begin{document}

\title{On the rationality of quadric surface bundles}
\author{Matthias Paulsen}
\address{Department Mathematisches Institut \\ Ludwig-Maximilians-Universität München \\ Theresienstrasse~39 \\ D-80333~München \\ Germany }
\email{paulsen@math.lmu.de}
\date{January~14, 2020}

\begin{abstract}
For any standard quadric surface bundle over $\PP^2$,
we show that the locus of rational fibres is dense in the moduli space.
\end{abstract}

\maketitle

\section{Introduction}

In \cite{hpt-quadrics}, Hassett, Pirutka, and Tschinkel gave the first example of a family $\Xcal\to B$ of smooth complex projective varieties such that
for a very general $b\in B$, the fibre $\Xcal_b$ is not stably rational, while the locus of $b\in B$ where $\Xcal_b$ is rational is dense in $B$ for the Euclidean topology.
Specifically, they considered the family of smooth complex hypersurfaces in $\PP^2\times\PP^3$ defined by a homogeneous polynomial of bidegree $(2,2)$.
Their result is remarkable as it shows that rationality of the fibres is in general not a closed property on the base. In particular, rationality is not deformation invariant in smooth families.

In order to prove stable irrationality of a very general member, they used the specialization method of Voisin \cite{voisin-unirational} and Colliot-Thélène--Pirutka \cite{ct-pirutka},
which allowed to disprove stable rationality in several other families as well, see e.\,g.\ \cite{voisin-survey} for an overview.

Subsequently, other smooth families containing both rational and stably irrational fibres were identified, for example in \cite{hpt-double}, \cite{hpt-intersection}, \cite{schreieder-long}, \cite{schreieder}, \cite{abp}, and \cite{hkt}.
Typically, it is easy to provide certain rational members in the studied families.
However, this does not exclude that the locus of rational fibres is contained in a proper closed subset of the base.
In only a few cases, it was shown that the locus of rational fibres is dense in the moduli space.

The fourfolds considered in \cite{hpt-quadrics} and \cite{hpt-double} are (birational to) quadric surface bundles over $\PP^2$
of types $(2,2,2,2)$ and $(0,2,2,4)$, respectively.
Here, a quadric surface bundle of type $(d_0,d_1,d_2,d_3)$ for integers $d_0,d_1,d_2,d_3\ge0$ of the same parity is given by an equation of the form
\begin{equation}
\sum_{0\le i,j\le3}a_{ij}y_iy_j = 0 \label{eq:quadratic-form}
\end{equation}
where $a_{ij}=a_{ji}$ is a homogeneous polynomial of degree $\frac12(d_i+d_j)$ in the three coordinates of $\PP^2$
and $y_0,y_1,y_2,y_3$ denote local trivializations of a split vector bundle $\Ecal$ on $\PP^2$ of rank~4,
see Section~\ref{sect:toric} for a more precise definition. The quadric surface bundle $X\subset\PP(\Ecal)$ over $\PP^2$
defined by equation~\eqref{eq:quadratic-form} is also called a \emph{standard quadric surface bundle}.
Apart from the examples in \cite{hpt-quadrics} and \cite{hpt-double},
many other fourfolds are birational to standard quadric surface bundles. For instance, 
a hypersurface in $\PP^5$ of degree~$d+2$ with multiplicity~$d$ along a plane for some integer $d\ge1$
is birational to a quadric surface bundle of type $(d,d,d,d+2)$, see e.\,g.\ \cite[Lemma~23]{schreieder-long}.

The smooth quadric surface bundles of fixed type $(d_0,d_1,d_2,d_3)$ are parametrized by a non-empty Zariski open subset $B\subset\PP(V)$ in
the projectivization of the complex vector space
\begin{equation}
V = \bigoplus_{0\le i\le j\le3}H^0\left(\PP^2,\Ocal_{\PP^2}\left(\tfrac12(d_i+d_j)\right)\right) \;. \label{eq:global}
\end{equation}
We may then consider the universal family $\Xcal\to B$ of smooth quadric surface bundles over $\PP^2$ of type $(d_0,d_1,d_2,d_3)$.

Using his improvement \cite{schreieder-long} of the specialization method, Schreieder
proved in \cite{schreieder} that a very general quadric surface bundle of type $(d_0,d_1,d_2,d_3)$ is not stably rational except for the two cases
$(1,1,1,3)$ and $(0,2,2,2)$ (up to reordering) which remain open and for trivial cases where the quadric surface bundle always has a rational section and is hence rational.
This vastly generalizes the irrationality results of \cite{hpt-quadrics} and \cite{hpt-double} to a natural class of families of quadric surface bundles over $\PP^2$.

The aim of this article is to prove the corresponding density assertion for any standard quadric surface bundle over $\PP^2$,
thus showing that in this large class of families the locus of rational fibres is never contained in a proper closed subset of the moduli space.
Concretely, we will prove the following:
\begin{theorem}\label{thm:main}
Let $d_0,d_1,d_2,d_3\ge0$ be integers of the same parity and let $\Xcal\to B\subset\PP(V)$ be the family of smooth quadric surface bundles over $\PP^2$ of type $(d_0,d_1,d_2,d_3)$ as above.
Then the set \[ \{b\in B \mid\Xcal_b\text{ is rational}\} \] is dense in $B$ for the Euclidean topology.
\end{theorem}
The first case where such a density result was proven was for type $(0,2,2,4)$ and is due to Voisin \cite[Section~2]{voisin-rationality}, see also \cite[Proposition~25]{schreieder-long}.
The case of type $(2,2,2,2)$ was shown in \cite{hpt-quadrics}.
In particular, Theorem~\ref{thm:main} generalizes their density result to hypersurfaces in $\PP^2\times\PP^3$ of bidegree $(d,2)$ for arbitrary $d\ge0$.
Our result also gives an affirmative answer to the question raised in \cite[Remark~49]{schreieder-long}.

In order to prove Theorem~\ref{thm:main}, we follow Voisin's approach that has later been used in \cite[Section~6]{hpt-quadrics} and \cite[Section~2.3]{hpt-intersection}.
Using a theorem of Springer \cite{springer} and the fact that the integral Hodge conjecture is known in codimension two for quadric bundles over surfaces \cite[Corollaire~8.2]{ct-voisin},
we obtain a Hodge theoretic criterion guaranteeing the rationality of smooth quadric surface bundles over $\PP^2$.
This leads to the study of a Noether--Lefschetz locus in the variation of Hodge structure associated to the family $\Xcal\to B$ in question.
In \cite[Proposition~5.20]{voisin2}, Voisin stated an infinitesimal condition for the density of such loci, based on Green's proof in \cite[Section~5]{chm} of an analogous density result in the context of the Noether--Lefschetz theorem.
In our case, the criterion asks for a class $\ol\lambda\in H^{2,2}_\van(\Xcal_b)$ at some base point $b\in B$ such that the infinitesimal period map evaluated at $\ol\lambda$
\[ \ol\nabla_b(\ol\lambda)\colon T_{B,b} \to H^{1,3}_\van(\Xcal_b) \]
is surjective.

Since a standard quadric surface bundle over $\PP^2$ is a toric variety,
we can apply \cite[Theorem~10.13]{batyrev-cox} to describe $\ol\nabla_b(\ol\lambda)$ as a multiplication map in a homogeneous quotient of a bigraded polynomial ring.
Therefore, the desired density result reduces to an elementary statement about polynomials.
This problem was solved in \cite{hpt-quadrics} and \cite{hpt-intersection} with explicit computations.
Of course, a different technique is required to handle a whole class of families rather than a specific one.
The main contribution of this paper consists thus in solving this problem to which Theorem~\ref{thm:main} reduces to via general arguments.
An important ingredient of our proof is a result about the strong Lefschetz property of certain complete intersections which was proven in \cite[Proposition~30]{harima-watanabe}.

Green's and Voisin's infinitesimal density criterion has been employed in many different situations since its first use in \cite[Section~5]{chm}.
For instance, Voisin used it in \cite{voisin-integral} when proving the integral Hodge conjecture for $(2,2)$-classes on uniruled or Calabi--Yau threefolds.
More recently, a real analogue of the criterion was applied in \cite{benoist} to prove that sums of three squares are dense among bivariate positive semidefinite real polynomials.

There exist different strategies for verifying the surjectivity of the infinitesimal period map.
While \cite{benoist} follows the approach of \cite{ciliberto-lopez} by constructing components of the Noether--Lefschetz locus of maximal codimension,
Kim gave in \cite[Theorem~2]{kim} a new proof of the density theorem from \cite[Section~5]{chm} by proving a statement about the Jacobian rings appearing in the description of $\ol\nabla_b(\ol\lambda)$.
The most general arguments are due to Voisin, for example in \cite{voisin-griffiths} and \cite{voisin-integral}.

We use the method of computing the infinitesimal period map explicitly, as done in \cite{kim}.
However, we solve the underlying algebraic problem in a different manner than in \cite[Section~3]{kim}.
Our approach involving the strong Lefschetz property, the use of which seems to be new in this area,
further allows to give a short proof for the density of the original Noether--Lefschetz locus for surfaces in $\PP^3$, thus simplifying the arguments of \cite{kim} considerably.

The article is structured as follows.
In Section~\ref{sect:criterion}, we relate the rationality of smooth quadric surface bundles over $\PP^2$ to the cohomology group $H^{2,2}$
and explain how Green's and Voisin's infinitesimal density criterion applies in our situation.
In Section~\ref{sect:toric}, we interpret standard quadric surface bundles as toric hypersurfaces in order to give an explicit representation of $\ol\nabla_b(\ol\lambda)$.
This cumulates in Proposition~\ref{prop:main}, where we formulate a non-trivial statement concerning a bigraded polynomial ring
which is sufficient for showing Theorem~\ref{thm:main}.
In Section~\ref{sect:preparations}, we provide some tools for studying the surjectivity of polynomial multiplication maps and demonstrate their power
by giving a simple proof for the density of the classical Noether--Lefschetz locus.
Finally, in Section~\ref{sect:proof} we use the previous preparations in order to prove Proposition~\ref{prop:main}, from which our main result follows.

Unless otherwise stated, we always work over the field of complex numbers.
A variety is defined to be an integral separated scheme of finite type over a field.
A quadric surface bundle over $\PP^2$ is a complex projective variety $X$ together with a flat morphism $\pi\colon X\to\PP^2$
such that the generic fibre $X_\eta$ is a smooth quadric surface over the function field $\CC(\PP^2)$.
If $X$ is a smooth complex projective variety and $Z\subset X$ is a subvariety of codimension $k$, we denote by $[Z]\in H^{k,k}(X,\ZZ)$
the Poincaré dual of the homology class of $Z$.

\section*{Acknowledgements}

I am very grateful to Stefan Schreieder for giving me the opportunity to work on this interesting question
and for many helpful comments concerning this article.
Further, I would like to thank an anonymous referee for his very careful reading and valuable suggestions improving the paper.

\section{A Density Criterion}\label{sect:criterion}

Let us consider a smooth quadric surface bundle $\pi\colon X\to\PP^2$.
Since $\PP^2$ is rational, $X$ is rational (over $\CC$) as soon as the generic fibre $X_\eta$ is rational over the function field $k=\CC(\PP^2)$.
It is well known that this follows from the existence of a $k$-point on the smooth quadric surface $X_\eta$.
Now we can use the following theorem of Springer \cite{springer}:
\begin{proposition}[Springer]\label{prop:springer}
Let $Q$ be a quadric hypersurface over a field $k$ and let $K/k$ be a finite field extension of odd degree.
If $Q$ has a $K$-point, then $Q$ has a $k$-point.
\end{proposition}
It therefore suffices to find a $K$-point on $X_\eta$ for some field extension $K/k$ of odd degree.
This can be achieved through an odd degree multisection of $\pi$,
i.\,e.\ a surface $Z\subset X$
such that $[Z]\cup\pi^*[p]\in H^{4,4}(X,\ZZ)\cong\ZZ$ is odd where $[p]\in H^{2,2}(\PP^2,\ZZ)\cong\ZZ$ denotes the cohomology class of a closed point,
since the function field $K=\CC(Z)$ is such a field extension then.

The integral Hodge conjecture was proven for $(2,2)$-classes on quadric bundles over surfaces by Colliot-Thélène and Voisin \cite[Corollaire~8.2]{ct-voisin}.
We use the following special case:
\begin{proposition}[Colliot-Thélène--Voisin]\label{prop:hodge-conjecture}
Let $\pi\colon X\to\PP^2$ be a smooth quadric surface bundle.
Then the integral Hodge conjecture holds for $H^{2,2}(X,\ZZ)$, i.\,e.\ any integral Hodge class $\alpha\in H^{2,2}(X,\ZZ)$
is an integral linear combination $\alpha=\sum n_i[Z_i]$ for surfaces $Z_i\subset X$.
\end{proposition}
This allows us to transform the assertion of $\pi$ having an odd degree multisection into a Hodge theoretic condition (see also \cite[Proposition~6]{hpt-quadrics}):
\begin{corollary}\label{cor:rationality}
Let $\pi\colon X\to\PP^2$ be a smooth quadric surface bundle.
Then $X$ is rational if there exists an integral Hodge class
$\alpha\in H^{2,2}(X,\ZZ)$ such that $\alpha\cup\pi^*[p]$ is odd.
\end{corollary}

Now let us consider the family $\Xcal\to B$ of smooth quadric surface bundles over $\PP^2$ of type $(d_0,d_1,d_2,d_3)$ for
fixed integers $d_j\ge0$ of the same parity.
In order to prove Theorem~\ref{thm:main}, it is enough by Corollary~\ref{cor:rationality} to show that the Noether--Lefschetz locus
\[ \{ b\in B \mid \exists\alpha\in H^{2,2}(\Xcal_b,\ZZ)\colon\alpha\cup\pi_b^*[p]\equiv1\pmod2 \} \]
is dense in $B$ for the Euclidean topology, where $\pi_b\colon\Xcal_b\to\PP^2$ denotes the quadric bundle structure on the fibre $\Xcal_b$.

Since it is easier to compute, we consider instead the \emph{vanishing cohomology}
\[ H^4_\van(\Xcal_b,\CC)=\{\alpha\in H^4(\Xcal_b,\CC)\mid\alpha\cup\iota^*\beta=0\ \forall\beta\in H^4(\PP(\Ecal),\CC)\} \]
where the map $\iota^*\colon H^4(\PP(\Ecal),\CC)\hookrightarrow H^4(\Xcal_b,\CC)$ is induced by inclusion and
is injective by the Lefschetz hyperplane theorem, provided that not all $d_j$ are simultaneously zero\footnote{If
$d_j=0$ for all $j$, Theorem~\ref{thm:main} is trivial because a quadric surface bundle of type $(0,0,0,0)$ is the product
of $\PP^2$ with a quadric surface in $\PP^3$ and hence rational.}.
This construction is also applicable to the Hodge groups $H^{p,q}$ and gives a decomposition
\[ H^4_\van(\Xcal_b,\CC)=\bigoplus_{p+q=4}H^{p,q}_\van(\Xcal_b) \;. \] We then want to show that the possibly smaller locus
\begin{equation}
\{ b\in B \mid \exists\alpha\in H^{2,2}_\van(\Xcal_b,\ZZ)\colon\alpha\cup\pi_b^*[p]\equiv1\pmod2 \} \label{eq:locus-van}
\end{equation}
is dense in $B$ for the Euclidean topology.
To achieve this, we utilise a variant of Voisin's description in \cite[Proposition~5.20]{voisin2} of an infinitesimal density criterion due to Green \cite[Section~5]{chm}.

On $B$ we consider the holomorphic vector bundle $\Hcal$ with fibre $\Hcal_b=H^4_\van(\Xcal_b,\CC)$ at $b\in B$.
By Ehresmann's lemma, $\Hcal$ is trivial over any contractible open subset of $B$.
The vector bundle $\Hcal$ is flat with respect to the Gauß--Manin connection $\nabla\colon\Hcal\to\Hcal\otimes\Omega_B$.
Since $H^{4,0}_\van(\Xcal_b)=H^{0,4}_\van(\Xcal_b)=0$ for all $b\in B$,
each fibre of $\Hcal$ has a Hodge filtration of weight~$2$.
It is well known that the Hodge filtration on the fibres of $\Hcal$ induces a filtration
\[ F^2\Hcal \subset F^1\Hcal \subset F^0\Hcal=\Hcal \]
by holomorphic subbundles. These satisfy Griffiths' transversality condition
\[ \nabla\left(F^p\Hcal^k\right) \subset F^{p-1}\Hcal^k\otimes\Omega_B \] for all $p$ and hence $\nabla$ gives rise to an $\Ocal_B$-linear map
\[ \ol\nabla\colon \Hcal^{1,1} \to \Hcal^{0,2}\otimes\Omega_B \]
on the quotients $\Hcal^{p,2-p}=F^p\Hcal/F^{p+1}\Hcal$.
Fibrewise, we obtain by adjunction the infinitesimal period map
\[ \ol\nabla_b\colon T_{B,b}\to\Hom\left(\Hcal^{1,1}_b,\Hcal^{0,2}_b\right) \]
for all $b\in B$.
Note that we may identify $\Hcal^{p,q}_b$ with $H^{p+1,q+1}_\van(\Xcal_b)$ for $p+q=2$.

Let $\Hcal_\RR$ be the real vector bundle on $B$ with fibre $\Hcal_{\RR,b}=H^4_\van(\Xcal_b,\RR)$ at $b\in B$.
Then we have $\Hcal_b=\Hcal_{\RR,b}\otimes_\RR\CC$ for all $b\in B$.
Similarly, for the real vector subbundle \[ \Hcal^{1,1}_\RR=\Hcal_\RR\cap F^1\Hcal\subset\Hcal_\RR \]
with fibre $\Hcal^{1,1}_{\RR,b}=H^{2,2}_\van(\Xcal_b,\RR)$ at $b\in B$
we have $\Hcal^{1,1}_b\cong\Hcal^{1,1}_{\RR,b}\otimes_\RR\CC$ for all $b\in B$.
The last identification is given by the restricted projection \[ p\colon\Hcal^{1,1}_\RR\subset F^1\Hcal\to F^1\Hcal/F^2\Hcal=\Hcal^{1,1} \;. \]
For all $b\in B$, let us consider the discrete subset
\[ D\Hcal_b = \{ \alpha\in H^4_\van(\Xcal_b,\ZZ) \mid \alpha\cup\pi_b^*[p]\equiv1\pmod2 \} \subset \Hcal_{\RR,b} \;. \]
Since $D\Hcal_b$ is defined by a topological property of $\Xcal_b$ which is compatible with the local trivializations of $\Xcal$ from Ehresmann's lemma
(it does in particular not depend on the Hodge filtration on $\Hcal_b$), we obtain a fibre subbundle $D\Hcal\subset\Hcal_\RR$ which is trivial over any contractible open subset of $B$.
Note that the locus \eqref{eq:locus-van} is precisely the image of the projection map $D\Hcal\cap\Hcal^{1,1}_\RR\to B$.
Our variant of \cite[Proposition~5.20]{voisin2} can now be stated as follows:
\begin{proposition}[Green--Voisin]\label{prop:criterion}
Suppose there exists $b\in B$ and $\ol\lambda\in\Hcal^{1,1}_b$ such that
the infinitesimal period map evaluated at $\ol\lambda$
\[ \ol\nabla_b(\ol\lambda)\colon T_{B,b}\to\Hcal^{0,2}_b \] is surjective.
Then the projection of $D\Hcal\cap\Hcal^{1,1}_\RR$ is dense in $B$ for the Euclidean topology.
\end{proposition}
\begin{proof}
We first observe that the surjectivity condition is a Zariski open property on $\ol\lambda\in\Hcal^{1,1}=\Hcal^{1,1}_\RR\otimes_\RR\CC$.
Hence, the condition is fulfilled on a dense open subset of the real classes $p(\Hcal^{1,1}_\RR)\subset\Hcal^{1,1}$.
Therefore, it suffices to show the statement locally around $b\in B$ where $\ol\lambda=p(\lambda)$ satisfies the hypothesis for some $\lambda\in\Hcal^{1,1}_{\RR,b}$.
By shrinking $B$, we may assume that the vector bundle $\Hcal_\RR$ is trivial over $B$, i.\,e.\ $\Hcal_\RR\cong B\times\Hcal_{\RR,b}$.
By \cite[Lemma~5.21]{voisin2}, the composed map
\[ \phi\colon \Hcal^{1,1}_\RR \hookrightarrow \Hcal_\RR \cong B\times\Hcal_{\RR,b} \to \Hcal_{\RR,b} \]
obtained via inclusion, isomorphism and projection is a submersion at $\lambda\in\Hcal^{1,1}_\RR$.
As shown in \cite[Lemma~20]{schreieder-long}, there are smooth quadric surface bundles $\Xcal_u$ of type $(d_0,d_1,d_2,d_3)$ which admit a rational section and hence $D\Hcal_u\ne\emptyset$.
Since $B$ is connected, it follows that $D\Hcal_b\ne\emptyset$.
By definition, $D\Hcal_b$ is a coset of a subgroup of $H^4_\van(\Xcal_b,\ZZ)$ of index~$2$.
Therefore, $\RR^*D\Hcal_b$ is dense in $\Hcal_{\RR,b}=H^4_\van(\Xcal_b,\ZZ)\otimes\RR$.
Since $\phi$ is a submersion,
the preimage $\phi^{-1}(\RR^*D\Hcal_b)$ is dense around $\lambda\in\Hcal^{1,1}_\RR$. But this precisely means $(\RR^*D\Hcal)\cap\Hcal^{1,1}_\RR$ is dense in $\Hcal^{1,1}_\RR$ around $\lambda$.
Hence, its projection is dense around $b\in B$. But the projections of $D\Hcal\cap\Hcal^{1,1}_\RR$ and $(\RR^*D\Hcal)\cap\Hcal^{1,1}_\RR$ agree because $\Hcal^{1,1}_\RR$ is a real vector bundle.
\end{proof}
Actually, the above proof works for any fibre bundle $D\Hcal\subset\Hcal_\RR$, trivial over contractible open subsets of $B$, such that $\RR^*D\Hcal_b$ is dense in $\Hcal_{\RR,b}$ for some $b\in B$.
This leads to a more general version of Proposition~\ref{prop:criterion}, which can be found in \cite[Section~3.3]{mythesis}.

\section{Computation of the Cohomology}\label{sect:toric}

We first give a more precise definition of standard quadric surface bundles over $\PP^2$, following \cite[Section~3.5]{schreieder-long}.
Let \[ \Ecal=\bigoplus_{j=0}^3\Ocal_{\PP^2}(-r_j) \] be a split vector bundle on $\PP^2$ for integers $r_j\ge0$
and let $q\colon\Ecal\to\Ocal_{\PP^2}(d)$ be a quadratic form for some integer $d\ge0$, i.\,e.\ a global section of
$\Sym^2\Ecal^\vee\otimes\Ocal_{\PP^2}(d)$.
Let us assume that the quadratic form $q_\eta$ at the generic point $\eta\in\PP^2$ is non-degenerate and that $q_s\ne0$ for all $s\in\PP^2$.
Then the zero set $X\subset\PP(\Ecal)$ of $q$ is a quadric surface bundle over $\PP^2$.
Since the vector bundle $\Sym^2\Ecal^\vee\otimes\Ocal_{\PP^2}(d)$ only depends on the integers $d_j=2r_j+d$ for $j\in\{0,1,2,3\}$,
we call $X$ a standard quadric surface bundle of type $(d_0,d_1,d_2,d_3)$.
Conversely, quadric surface bundles of type $(d_0,d_1,d_2,d_3)$ for given integers $d_j\ge0$
exist whenever $d_0,d_1,d_2,d_3$ are of the same parity\footnote{One can always ensure $d\in\{0,1\}$, but this is not needed in our arguments.}.
Since
\[ H^0\left(\PP^2,\Sym^2\Ecal^\vee\otimes\Ocal_{\PP^2}(d)\right) \cong \bigoplus_{0\le i\le j\le 3}H^0\left(\PP^2,\Ocal_{\PP^2}(r_i)\otimes\Ocal_{\PP^2}(r_j)\otimes\Ocal_{\PP^2}(d)\right) = V \]
where $V$ was defined in \eqref{eq:global}, $X$ can be described by an equation of the form \eqref{eq:quadratic-form} where $y_j$ is a local
trivialization of $\Ocal_{\PP^2}(-r_j)$.

We now aim to interpret \eqref{eq:quadratic-form} differently as a global equation inside the polynomial ring 
\[ S = \CC[x_0,x_1,x_2;y_0,y_1,y_2,y_3] \] endowed with a non-standard bigrading.
By \cite[Example~7.3.5]{cls-toric}, the total space $\PP(\Ecal)$ is a toric variety
associated to a fan $\Sigma$ in $\RR^2\times\RR^3$ and has coordinate ring $S$.
If $u_1,u_2$ and $v_1,v_2,v_3$ denote the standard basis vectors of $\RR^2$ and $\RR^3$, respectively, then the seven $1$-dimensional cones of $\Sigma$
are generated by $u_0,u_1,u_2,v_0,v_1,v_2,v_3$ where
\[ u_0 = -\sum_{k=1}^2u_k+\sum_{j=1}^3(r_j-r_0)v_j \quad\text{and}\quad v_0 = -\sum_{j=1}^3v_j \;. \]
Further, the maximal cones of $\Sigma$ are given by
\[ \left< u_0,\ldots,\hat u_k,\ldots,u_2, v_0,\ldots,\hat v_j,\ldots,v_3 \right> \;,\quad k\in\{0,1,2\}\;,\quad j\in\{0,1,2,3\} \;. \]
By \cite[Definition~1.7]{batyrev-cox}, we have $\Cl(\Sigma)\cong\ZZ^7/\Im C$ where
\[ C=\begin{pmatrix} -1 & -1 & r_1-r_0 & r_2-r_0 & r_3-r_0 \\ 1 & 0 & 0 & 0 & 0 \\ 0 & 1 & 0 & 0 & 0 \\ 0 & 0 & -1 & -1 & -1 \\
0 & 0 & 1 & 0 & 0 \\ 0 & 0 & 0 & 1 & 0 \\ 0 & 0 & 0 & 0 & 1 \end{pmatrix} \in \Hom(\ZZ^5,\ZZ^7) \;. \]
It is easy to see that the surjection
\begin{align*}
\ZZ^7 &\to \ZZ^2 \\
(m_0,m_1,m_2,n_0,n_1,n_2,n_3) &\mapsto \left(\sum_{k=0}^2m_k-\sum_{j=0}^3n_jr_j,\sum_{j=0}^3n_j\right)
\end{align*}
has kernel $\Im C$. Hence, this map descends to an isomorphism $\Cl(\Sigma)\cong\ZZ^2$ and endowes the coordinate ring $S$ with the non-standard bigrading
\[ \deg x_k=(1,0)\;,\quad\deg y_j=(-r_j,1) \]
for $k\in\{0,1,2\}$ and $j\in\{0,1,2,3\}$.
For $m,n\in\ZZ$, we denote by $S(m,n)$ the subspace of homogeneous polynomials of bidegree $(m,n)$ in $S$.
This gives a decomposition
\[ S = \bigoplus_{m,n\in\ZZ}S(m,n) \]
into finite dimensional $\CC$-vector spaces.

A quadratic form $q\colon\Ecal\to\Ocal_{\PP^2}(d)$ corresponds to an element in $S(d,2)$.
In this way, the local description \eqref{eq:quadratic-form} of the zero set of $q$ can be seen globally
as a defining equation for a toric hypersurface $X\subset\PP(\Ecal)$.

This allows us to compute the middle cohomology groups of a smooth quadric surface bundle $\pi\colon X\to\PP^2$ of type $(d_0,d_1,d_2,d_3)$
defined by a polynomial $f\in S(d,2)$ via the method of \cite[Theorem~10.13]{batyrev-cox},
which generalizes the work of Griffiths \cite{griffiths69} to toric hypersurfaces.
We have canonical isomorphisms
\[ H^{1,3}_\van(X)\cong R(t,4) \quad\text{and}\quad H^{2,2}_\van(X)\cong R(t-d,2) \]
where \[ t=4d-3+r_0+r_1+r_2+r_3 \] and where $R$ denotes the Jacobian ring of $f$,
i.\,e.\ the quotient of $S$ by all partial derivatives of $f$.

Now we return to the family $\Xcal\to B$ of smooth quadric surface bundles of type $(d_0,d_1,d_2,d_3)$.
If we identify $T_{B,b}\cong(S/fS)(d,2)$ where $f\in S(d,2)$ is the defining equation of $\Xcal_b$ for some $b\in B$,
then the infinitesimal period map
\[ \ol\nabla_b\colon T_{B,b}\otimes H^{2,2}_\van(\Xcal_b) \to H^{1,3}_\van(\Xcal_b) \]
is given, up to a sign, as the multiplication map
\[ (S/fS)(d,2)\otimes R(t-d,2) \to R(t,4) \;. \]
This was first shown for hypersurfaces in projective space by Carlson and Griffiths \cite{carlson-griffiths}, see also \cite[Theorem~6.13]{voisin2}.
In order to show that the assumption of Proposition~\ref{prop:criterion} holds and thus to prove Theorem~\ref{thm:main},
it therefore suffices to provide polynomials $f\in S(d,2)$ and $g\in S(t-d,2)$ such that the quadric surface bundle $\{f=0\}\subset\PP(\Ecal)$ is smooth
and the composed map $S(d,2)\to R(t,4)$ given by multiplication with $g$ followed by projection is surjective.
By Bertini's theorem, the hypersurface $\{f=0\}\subset\PP(\Ecal)$ is smooth for a general polynomial $f\in S(d,2)$.
The surjectivity part is equivalent to claiming that the ideal generated by $g$ and all partial derivatives of $f$ contains all polynomials in $S(t,4)$.
Consequently, we reduced Theorem~\ref{thm:main} to the following statement:
\begin{proposition}\label{prop:main}
For general polynomials $f\in S(d,2)$ and $g\in S(t-d,2)$, the ideal in $S$ generated by the polynomials
\[ \frac{\partial f}{\partial x_0}\;,\quad\frac{\partial f}{\partial x_1}\;,\quad\frac{\partial f}{\partial x_2}\;,\quad
\frac{\partial f}{\partial y_0}\;,\quad\frac{\partial f}{\partial y_1}\;,\quad\frac{\partial f}{\partial y_2}\;,\quad\frac{\partial f}{\partial y_3}\;,\quad g \]
contains all polynomials in $S(t,4)$.
\end{proposition}
The remaining part of the paper is devoted to the proof of this proposition.

\section{Preparations}\label{sect:preparations}

The property that a homogeneous ideal in a bigraded polynomial ring (or more generally, an arbitrarily graded $\CC$-algebra) contains all polynomials of a certain bidegree is,
as we now show, a Zariski open condition on its generators if their bidegrees are fixed.
\begin{lemma}\label{lemma:open}
Let $G$ be an Abelian group and let $A$ be a $G$-graded $\CC$-algebra whose homogeneous components $A(m)$ are finite dimensional $\CC$-vector spaces for all $m\in G$.
Let $m_0,\ldots,m_k\in G$. Then the set
\[ \left\{(f_1,\ldots,f_k)\in A(m_1)\oplus\cdots\oplus A(m_k)\mid A(m_0)\subset f_1A+\cdots+f_kA\right\} \]
is Zariski open.
\end{lemma}
\begin{proof}
The condition on $(f_1,\ldots,f_k)$ is equivalent to saying that the $\CC$-linear map
\begin{align*}
A(m_0-m_1)\oplus\cdots\oplus A(m_0-m_k) &\to A(m_0) \\
(g_1,\ldots,g_k) &\mapsto f_1g_1+\cdots+f_kg_k
\end{align*}
is surjective. This map is represented by a matrix $B$ with $r=\dim_\CC A(m_0)$ rows, whose entries are linear polynomials in the coefficients of $f_1,\ldots,f_k$.
The locus in $A(m_1)\oplus\cdots\oplus A(m_k)$ where this linear map is not surjective is precisely where the determinants of all $(r\times r)$-submatrices of $B$ vanish (in particular, it is the whole affine space if $B$ has less than $r$ columns) and thus Zariski closed.
Therefore, the set in question is open for the Zariski topology.
\end{proof}
Since taking partial derivatives is a linear and hence Zariski continuous map between the respective $\ZZ^2$-graded pieces of $S$, Lemma~\ref{lemma:open}
shows that the desired condition in Proposition~\ref{prop:main} is Zariski open on $f$ and $g$.

Apart from $S$, we will often apply Lemma~\ref{lemma:open} to the polynomial ring $\CC[x_0,x_1,x_2]$ together with its usual grading.
In this situation, we can give sufficient criteria whether three or four polynomials satisfy the Zariski open condition in the lemma.
More generally, for $n\ge0$ we can give such criteria for $n+1$ and $n+2$ polynomials in the graded polynomial ring
\[ P_n=\CC[x_0,\ldots,x_n]=\bigoplus_{m\ge0}P_n(m) \;. \]
\begin{lemma}\label{lemma:gen3}
If $f_0,\ldots,f_n\in P_n$ form a complete intersection, i.\,e.\ they have no common zero in $\PP^n$, then
\[ P_n(m) \subset f_0P_n + \cdots + f_nP_n \]
for all $m\ge m_0+\cdots+m_n-n$ where $f_j\in P_n(m_j)$ for $j\in\{0,\ldots,n\}$.
\end{lemma}
\begin{proof}
This immediately follows from Macaulay's Theorem (see for example \cite[Section~6.2.2]{voisin2}) which tells us that the quotient of
$P_n$ by the ideal generated by $f_0,\ldots,f_n$ is a graded Gorenstein ring with socle degree $\sum(m_j-1)$,
and hence its $m$-th graded piece is zero-dimensional for all $m\ge\sum m_j-n$.
\end{proof}
To state a sufficient criterion whether $n+2$~polynomials in $P_n$ belong to the Zariski open set in Lemma~\ref{lemma:open},
we use the so called \emph{strong Lefschetz property}, see e.\,g.\ \cite{stanley}. A quotient $Q$ of $P_n$ by
homogeneous polynomials $f_0,\ldots,f_n\in P_n$ is said to have the strong Lefschetz property if there exists a linear homogeneous polynomial
$\ell\in P_n(1)$ such that the map $Q(m)\to Q(m+i)$ given by multiplication with $\ell^i$ has maximal rank for all $m,i\ge0$.
The polynomial $\ell$ is then called a \emph{strong Lefschetz element} for the system $f_0,\ldots,f_n$.
\begin{lemma}\label{lemma:gen4}
If $f_0,\ldots,f_n\in P_n$ form a complete intersection having the strong Lefschetz property
and $f_{n+1}\in P_n$ is a power of a strong Lefschetz element for $f_0,\ldots,f_n$, then
\[ P_n(m) \subset f_0P_n + \cdots + f_{n+1}P_n \]
for all $m\ge\frac12(m_0+\cdots+m_{n+1}-n-1)$ where $f_j\in P_n(m_j)$ for $j\in\{0,\ldots,n+1\}$.
\end{lemma}
\begin{proof}
As in Lemma~\ref{lemma:gen3}, the quotient $Q$ of $P_n$ by $f_0,\ldots,f_n$ is a graded Gorenstein ring with socle degree
$s=\sum(m_j-1)$. Macaulay's Theorem also shows that $\dim_\CC Q(i)=\dim_\CC Q(s-i)$ for all $i\in\ZZ$.
Because of the strong Lefschetz property, $\dim_\CC Q(i)$ needs to be increasing for $i\le\frac s2$
and decreasing for $i\ge\frac s2$.
The claimed statement is equivalent to saying that the map $Q(m-m_{n+1})\to Q(m)$ given by multiplication with $f_{n+1}$ is surjective.
Since $f_{n+1}$ is a power of a strong Lefschetz element, it suffices to show $\dim_\CC Q(m-m_{n+1})\ge\dim_\CC Q(m)$.
This is clear if $m-m_{n+1}\ge\frac s2$. For $m-m_{n+1}\le\frac s2$, we have $\dim_\CC Q(m)=\dim_\CC Q(s-m)\le\dim_\CC Q(m-m_{n+1})$
because $s-m\le m-m_{n+1}$ holds due to the given bound on $m$.
\end{proof}
To make use of Lemma~\ref{lemma:gen4}, it is convenient to have a rich source of complete intersections enjoying the strong Lefschetz property.
The following important result, proved in 1980 by Stanley \cite{stanley} and independently in 1987 by Watanabe \cite{watanabe},
was the starting point for the theory of Lefschetz properties:
\begin{proposition}[Stanley--Watanabe]\label{prop:stanley}
A monomial complete intersection $x_0^{m_0},\ldots,x_n^{m_n}$ in $P_n$ with $m_0,\ldots,m_n\ge0$ has the strong Lefschetz property for all $n\ge0$.
\end{proposition}
Stanley's proof goes as follows: If we interpret the graded quotient $Q=\bigoplus_{m\ge0}Q(m)$ of $P_n$
by the monomials $x_0^{m_0},\ldots,x_n^{m_n}$
as the cohomology ring in even degree $H^{2\bullet}(X,\CC)$ of the Kähler manifold $X=\PP^{m_0-1}\times\cdots\times\PP^{m_n-1}$,
the linear polynomial $\ell=x_1+\cdots+x_n$ corresponds to the cohomology class of a Kähler form on $X$
and the strong Lefschetz property for $\ell$ precisely translates into the hard Lefschetz theorem for $X$,
hence also the name of this condition.

It is known for $n\le1$ and conjectured for $n\ge2$ that actually all complete intersections in $P_n$ have the strong Lefschetz property.
For $n=2$, the following partial result proven in \cite[Proposition~30]{harima-watanabe} satisfies our needs for the proof of Proposition~\ref{prop:main}:
\begin{proposition}[Harima--Watanabe]\label{prop:lefschetz}
If $f_0,f_1,f_2\in P_2=\CC[x_0,x_1,x_2]$ form a complete intersection such that $f_0$ is a power of a linear polynomial,
then $f_0,f_1,f_2$ has the strong Lefschetz property.
\end{proposition}

As a motivating example, we show how Lemmas \ref{lemma:gen3} and \ref{lemma:gen4} can be used to give a short proof
for the density of the classical Noether--Lefschetz locus for surfaces in $\PP^3$. For this, we do not need
Proposition~\ref{prop:lefschetz}, but only the earlier result stated in Proposition~\ref{prop:stanley}.
Since the setup here is a lot easier than in the case of standard quadric surface bundles,
this will also be a good preparation for the more involved arguments in Section~\ref{sect:proof}.

\begin{theorem}[Ciliberto--Harris--Miranda--Green]
For $d\ge4$, let $\Xcal\to B\subset\PP(P_3(d))$ be the universal family of smooth surfaces in $\PP^3$ of degree~$d$.
Then the Noether--Lefschetz locus
\[ \left\{b\in B \;\middle|\;\Pic(\Xcal_b)\supsetneq\ZZ\cdot\Ocal_{\PP^3}(1)|_{\Xcal_b}\right\} = \left\{ b\in B \;\middle|\; H^{1,1}_\van(\Xcal_b,\ZZ)\ne0 \right\} \;, \]
i.\,e.\ those surfaces containing curves which are no complete intersections, is dense in $B$ for the Euclidean topology.
\end{theorem}
\begin{proof}
By Green's and Voisin's infinitesimal density criterion, it suffices to show that there exists a point $b\in B$ and a class $\ol\lambda\in H^{1,1}_\van(\Xcal_b)$
such that
\[ \ol\nabla_b(\ol\lambda)\colon T_{B,b} \to H^{0,2}_\van(\Xcal_b) \]
is surjective.
For a surface $X\subset\PP^3$ defined by a polynomial $f\in P_3(d)$, Griffiths \cite{griffiths69} has shown that
\[ H^{0,2}_\van(X) \cong R(3d-4) \quad\text{and}\quad H^{1,1}_\van(X) \cong R(2d-4) \]
where $R$ denotes the Jacobian ring of $f$, i.\,e.\ the quotient of $P_3$ by the partial derivatives of $f$.
If we identify $T_{B,b}\cong(P_3/fP_3)(d)$ where $f\in P_3(d)$ is the defining equation of $\Xcal_b$ for some $b\in B$,
Carlson and Griffiths \cite{carlson-griffiths} proved that the infinitesimal period map
\[ \ol\nabla_b\colon T_{B,b} \otimes H^{1,1}_\van(\Xcal_b) \to H^{0,2}_\van(\Xcal_b) \]
is given, up to a sign, as the multiplication map
\[ (P_3/fP_3)(d) \otimes R(2d-4) \to R(3d-4) \;. \]
Therefore, it suffices to find polynomials $f\in P_3(d)$ and $g\in P_3(2d-4)$
such that the surface $\{f=0\}\subset\PP^3$ is smooth and the ideal generated by $g$ and the partial derivatives of $f$
contains the whole of $P_3(3d-4)$.

One can achieve this with the smooth Fermat surface defined by \[ f=x_0^d+x_1^d+x_2^d+x_3^d \;, \]
which was also used in \cite[Section~3]{kim}.
Since the complete intersection consisting of the partial derivatives of $f$ has the strong Lefschetz property by Proposition~\ref{prop:stanley},
we can take $g$ to be a power of a corresponding strong Lefschetz element and obtain via Lemma~\ref{lemma:gen4}
\[ P_3(m) \subset x_0^{d-1}P_3+x_1^{d-1}P_3+x_2^{d-1}P_3+x_3^{d-1}P_3+gP_3 \]
for all $m\ge\frac12(4(d-1)+2d-4-4)=3d-6$.
Since $3d-4\ge3d-6$, this finishes the proof.
\end{proof}

\section{Proof of Proposition~\ref{prop:main}}\label{sect:proof}

Without loss of generality, let $r_0\le r_1\le r_2\le r_3$. Let us recall from Section~\ref{sect:toric}
that $d_j=2r_j+d$ for $j\in\{0,1,2,3\}$ and $t=4d-3+\sum r_j$.
By Lemma~\ref{lemma:open}, the property stated in Proposition~\ref{prop:main} is Zariski open on $f$ and $g$.
Hence, it suffices to show the existence of polynomials $f\in S(d,2)$ and $g\in S(t-d,2)$ such that the homogeneous ideal $I\subset S$ generated by
\[ \frac{\partial f}{\partial x_0}\;,\quad\frac{\partial f}{\partial x_1}\;,\quad\frac{\partial f}{\partial x_2}\;,\quad
\frac{\partial f}{\partial y_0}\;,\quad\frac{\partial f}{\partial y_1}\;,\quad\frac{\partial f}{\partial y_2}\;,\quad\frac{\partial f}{\partial y_3}\;,\quad g \]
contains all polynomials in $S(t,4)$.
Let
\[ f=f_0y_0^2+f_1y_1^2+f_2y_2^2+f_3y_3^2\in S(d,2) \]
where $f_j\in S(d_j,0)$ are general for $j\in\{0,1,2,3\}$. Further let
\[ g=g_{11}y_1^2+g_{33}y_3^2+\sum_{0\le i<j\le3}g_{ij}y_iy_j\in S(t-d,2) \]
where $g_{ij}\in S(t-d+r_i+r_j,0)$ are general for $i,j\in\{0,1,2,3\}$.
Instead of proving directly that $S(t,4)\subset I$, we will consider the homogeneous ideal
\[ J = \bigoplus_{m,n\in\ZZ} \{ r\in S(m,n)\mid rS\cap S(t,4)\subset I \} \;, \]
and aim to show $J=S$. One can think of $J$ as all relations which hold if a polynomial of bidegree $(t,4)$ is considered modulo~$I$.
Since $I\subset J$, the following congruences hold:
\begin{align}
f_jy_j &\equiv 0\pmod J\;,\quad j\in\{0,1,2,3\}\;, \label{eq:fy} \\
\frac{\partial f_0}{\partial x_k}y_0^2+\frac{\partial f_1}{\partial x_k}y_1^2+\frac{\partial f_2}{\partial x_k}y_2^2+\frac{\partial f_3}{\partial x_k}y_3^2 &\equiv 0\pmod J\;,\quad k\in\{0,1,2\}\;, \label{eq:dfy2} \\
g_{11}y_1^2+g_{33}y_3^2+\sum_{0\le i<j\le3}g_{ij}y_iy_j &\equiv 0\pmod J \;. \label{eq:g}
\end{align}
It suffices to show $S(t,4)\subset J$. For this it is enough to prove the following four claims for all permutations $\sigma$ of $\{0,1,2,3\}$:
\[ y_{\sigma(0)}y_{\sigma(1)}y_{\sigma(2)}\in J \;,\quad y_{\sigma(0)}^3y_{\sigma(1)}\in J \;,\quad y_{\sigma(0)}^2y_{\sigma(1)}^2\in J \;,\quad y_{\sigma(0)}^4\in J \;. \]
The proof of each of these claims will constitute one of the four steps \ref{step:stu}, \ref{step:s3t}, \ref{step:s2t2}, and \ref{step:s4} below.
In each step, it suffices to show that any monomial of bidegree $(t,4)$ containing the specified variables $y_j$ can be reduced to $0$ modulo~$J$ using the congruences \eqref{eq:fy}, \eqref{eq:dfy2}, \eqref{eq:g}, and the previous steps.
In fact, the assertion $r_0\le r_1\le r_2\le r_3$ and the congruence \eqref{eq:g} will not be used in the first two steps, so we are allowed to restrict ourselves to the case $\sigma=\id$ in these two steps.

\subsection{First step}\label{step:stu}

We have $y_{\sigma(0)}y_{\sigma(1)}y_{\sigma(2)}\in J$ for all permutations $\sigma$ of $\{0,1,2,3\}$.
\begin{proof}
Without loss of generality, let $\sigma=\id$.
We first note that
\begin{equation}
S(d_0+d_1+d_2-2,0)\subset f_0S+f_1S+f_2S \;. \label{eq:ci}
\end{equation}
This follows from Lemmas \ref{lemma:open} and \ref{lemma:gen3} because there are complete intersections $f_0,f_1,f_2$ in \[ P_2=\CC[x_0,x_1,x_2]=\bigoplus_{m\ge0}S(m,0) \;. \]
Now let us take a monomial $hy_0y_1y_2y_j\in S(t,4)$ where $j\in\{0,1,2,3\}$ and $h\in S(t+r_0+r_1+r_2+r_j,0)$.
We may assume that $r_j>0$ or $d>0$, since for $d_j=2r_j+d=0$ we have $y_j\equiv0\pmod J$ by \eqref{eq:fy} and hence $hy_0y_1y_2y_j\equiv0\pmod J$.
In view of \eqref{eq:fy} and \eqref{eq:ci}, it suffices to show that
\[ t+r_0+r_1+r_2+r_j \ge d_0+d_1+d_2-2 \;. \]
This is equivalent to
\[ 2r_0+2r_1+2r_2+r_3+r_j+4d-3 \ge 2r_0+2r_1+2r_2+3d-2 \]
or just $r_3+r_j+d \ge 1$,
which is true because $r_j>0$ or $d>0$.
\end{proof}

\subsection{Second step}\label{step:s3t}

We have $y_{\sigma(0)}^3y_{\sigma(1)}\in J$ for all permutations $\sigma$ of $\{0,1,2,3\}$.
\begin{proof}
Without loss of generality, let $\sigma=\id$.
Multiplying \eqref{eq:dfy2} with $y_0y_1$ and using step~\ref{step:stu} yields
\begin{equation}
\left(\frac{\partial f_0}{\partial x_k}y_0^2+\frac{\partial f_1}{\partial x_k}y_1^2\right)y_0y_1\equiv0 \pmod J\;,\quad k\in\{0,1,2\}\;. \label{eq:s3t+st3}
\end{equation}
We introduce the new polynomial ring $T=\CC[x_0,x_1,x_2;z_0,z_1]$ with the bigrading
\[ \deg x_k=(1,0)\;,\quad\deg z_j=(-d_j,1) \]
for $k\in\{0,1,2\}$ and $j\in\{0,1\}$.
\begin{claim}
We have
\begin{equation}
T(d_0+d_1-3,1) \subset f_0T+f_1T+\left(\frac{\partial f_0}{\partial x_0}z_0+\frac{\partial f_1}{\partial x_0}z_1\right)T+\left(\frac{\partial f_0}{\partial x_1}z_0+\frac{\partial f_1}{\partial x_1}z_1\right)T \;. \label{eq:claimT}
\end{equation}
\end{claim}
\begin{proof}[Proof of the claim]
The claim is true if $d_0=0$ or $d_1=0$ because $f_0$ or $f_1$ is a unit then. If $d_0,d_1>0$, setting $f_0=(x_0+x_1)^{d_0}+x_2^{d_0}$ and $f_1=(x_0-x_1)^{d_1}+x_2^{d_1}$ yields
\[ \left(\frac{\partial f_0}{\partial x_0}z_0+\frac{\partial f_1}{\partial x_0}z_1\right)+\left(\frac{\partial f_0}{\partial x_1}z_0+\frac{\partial f_1}{\partial x_1}z_1\right)=2d_0(x_0+x_1)^{d_0-1}z_0 \;. \]
Since $(x_0+x_1)^{d_0-1},f_0,f_1$ form a complete intersection in \[ P_2=\CC[x_0,x_1,x_2]=\bigoplus_{m\ge0}T(m,0) \;, \] Lemma~\ref{lemma:gen3} implies that the right-hand side of \eqref{eq:claimT} contains all polynomials in $T(d_0+d_1-3,1)$ of type $hz_0$ where $h\in T(2d_0+d_1-3,0)$.
Similarly,
\[ \left(\frac{\partial f_0}{\partial x_0}z_0+\frac{\partial f_1}{\partial x_0}z_1\right)-\left(\frac{\partial f_0}{\partial x_1}z_0+\frac{\partial f_1}{\partial x_1}z_1\right)=2d_1(x_0-x_1)^{d_1-1}z_1 \]
and $(x_0-x_1)^{d_1-1},f_0,f_1$ are again a complete intersection, so all polynomials in $T(d_0+d_1-3,1)$ divisible by $z_1$ are contained in the right-hand side of \eqref{eq:claimT} as well.
Hence, the claim follows from Lemma~\ref{lemma:open} applied the polynomial ring $T$, since the coefficients of the four polynomials which are supposed to generate $T(d_0+d_1-3,1)$ depend linearly and thus Zariski continuously on those of the general polynomials $f_0$ and $f_1$.
\end{proof}
Now let us take a monomial $hy_0^3y_1\in S(t,4)$ where $h\in S(t+3r_0+r_1,0)$. We have
\[ t+3r_0+r_1=4r_0+2r_1+r_2+r_3+4d-3 \ge 4r_0+2r_1+3d-3 = 2d_0+d_1-3 \;. \]
Therefore, as a consequence of \eqref{eq:claimT} we obtain
\[ hz_0 = h_0f_0 + h_1f_1 + h_2\left(\frac{\partial f_0}{\partial x_0}z_0+\frac{\partial f_1}{\partial x_0}z_1\right) + h_3\left(\frac{\partial f_0}{\partial x_1}z_0+\frac{\partial f_1}{\partial x_1}z_1\right) \]
for certain homogeneous polynomials $h_0,h_1,h_2,h_3\in T$. Substituting $z_j$ by $y_j^2$ for $j\in\{0,1\}$ and multiplying with $y_0y_1$, we get by \eqref{eq:fy} and \eqref{eq:s3t+st3}
\begin{align*}
hy_0^3y_1 &= \tilde h_0f_0y_0y_1 + \tilde h_1f_1y_0y_1 + h_2\left(\frac{\partial f_0}{\partial x_0}y_0^2+\frac{\partial f_1}{\partial x_0}y_1^2\right)y_0y_1 + h_3\left(\frac{\partial f_0}{\partial x_0}y_0^2+\frac{\partial f_1}{\partial x_0}y_1^2\right)y_0y_1 \\
&\equiv \tilde h_0y_1\cdot0 + \tilde h_1y_0\cdot0 + h_2\cdot0 + h_3\cdot0 \equiv 0 \pmod J
\end{align*}
where $\tilde h_0$ and $\tilde h_1$ denote the results of the substitution inside $h_0$ and $h_1$.
\end{proof}

\subsection{Third step}\label{step:s2t2}

We have $y_{\sigma(0)}^2y_{\sigma(1)}^2\in J$ for all permutations $\sigma$ of $\{0,1,2,3\}$.
\begin{proof}
Multiplying \eqref{eq:g} with $y_iy_j$ for $0\le i<j\le 3$ and using the previous steps,
we obtain
\begin{equation}
g_{ij}y_i^2y_j^2\equiv0\pmod J\;. \label{eq:gs2t2}
\end{equation}
For the following definition, we assume $d_0>0$ at first. For $j\in\{0,1,2,3\}$, let $\hat A_j$ be the $(3\times3)$-matrix where we leave out the $j$-th column (counted from $0$) of the matrix
\[ \left(\frac{\partial f_j}{\partial x_k}\right)_{\substack{k\in\{0,1,2\}\\\mathclap{j\in\{0,1,2,3\}}}} \;. \]
A straightforward calculation shows that \eqref{eq:dfy2} implies
\begin{equation}
\left(\det\hat A_j\right)y_i^2\equiv\eps_{ij}\left(\det\hat A_i\right)y_j^2\pmod J\;,\quad i,j\in\{0,1,2,3\} \label{eq:det}
\end{equation}
where $\det\hat A_j\in S(d_0+d_1+d_2+d_3-d_j-3,0)$ for $j\in\{0,1,2,3\}$ and $\eps_{ij}\in\{\pm1\}$ is a sign depending on $i,j\in\{0,1,2,3\}$.

For $d_0=0$, both sides of \eqref{eq:det} would be zero since $\frac{\partial f_0}{\partial x_k}=0$ for $k\in\{0,1,2\}$.
Therefore, in the case $d_0=0$ we define the matrix $\hat A_j$ for $j\in\{1,2,3\}$ to be the $(2\times2)$-matrix where one leaves out the $j$-th column (counted from $1$) of the matrix
\[ \left(\frac{\partial f_j}{\partial x_k}\right)_{\substack{k\in\{0,1\}\\\mathclap{j\in\{1,2,3\}}}} \;. \]
Because \eqref{eq:fy} implies $y_0\equiv0\pmod J$ in this case, one can still conclude from \eqref{eq:dfy2} that
\begin{equation}
\left(\det\hat A_j\right)y_i^2\equiv\eps_{ij}\left(\det\hat A_i\right)y_j^2\pmod J\;,\quad i,j\in\{1,2,3\} \label{eq:det0}
\end{equation}
where $\det\hat A_j\in S(d_1+d_2+d_3-d_j-2,0)$ for $j\in\{1,2,3\}$ and $\eps_{ij}\in\{\pm1\}$ may be different for $i,j\in\{1,2,3\}$.

Let us first suppose that $\{\sigma(0),\sigma(1)\}=\{1,2\}$.
Multiplying \eqref{eq:g} with $y_2^2$ and using steps \ref{step:stu} and \ref{step:s3t} yields
\begin{equation}
g_{11}y_1^2y_2^2+g_{33}y_2^2y_3^2\equiv0\pmod J\;. \label{eq:gs2t2+gt2u2}
\end{equation}
Let us consider the polynomial ring $U=\CC[x_0,x_1,x_2;z_1,z_3]$ with the bigrading
\[ \deg x_k=(1,0)\;,\quad\deg z_j=(-d_j,1) \]
for $k\in\{0,1,2\}$ and $j\in\{1,3\}$.
We claim that
\begin{equation}
U(t-d+2r_2,1) \subset K \;, \label{eq:claimU}
\end{equation}
where $K$ denotes the ideal in $U$ generated by
\[ f_1z_1\;,\quad f_2\;,\quad f_3z_3\;,\quad g_{12}z_1\;,\quad g_{23}z_3\;,\quad g_{11}z_1+g_{33}z_3\;,\quad\left(\det\hat A_3\right)z_1-\eps_{13}\left(\det\hat A_1\right)z_3 \;. \]
Since the coefficients of these seven polynomials in $U$ depend algebraically on those of $f_0,f_1,f_2,f_3,g_{11},g_{12},g_{23},g_{33}$,
Lemma~\ref{lemma:open} with $A=U$ shows that it is enough to provide a special choice for the general polynomials
$f_j,g_{ij}\in\CC[x_0,x_1,x_2]$ making \eqref{eq:claimU} true.
\begin{claim}
This can be achieved in the following way,
where $\mu,\nu\in U(1,0)$ denote suitable strong Lefschetz elements of complete intersections that will be specified later:
\begin{align*}
f_0 &= x_0^{d_0} & g_{11} &= x_2^{t-d+2r_1} \\
f_1 &= x_0^{d_1} & g_{12} &= \nu^{t-d+r_1+r_2} \\
f_2 &= x_0^{d_2}+x_1^{d_2} & g_{23} &= \mu^{t-d+r_2+r_3} \\
f_3 &= x_0^{d_3}+x_2^{d_3} & g_{33} &= \mu^{t-d+2r_3}
\end{align*}
\end{claim}
\begin{proof}[Proof of the claim]
The claim is obvious for $d_2=0$, so we may assume $d_2>0$ in the following.
As in the case of the ideal $I$, we consider instead the larger homogeneous ideal
\[ L = \bigoplus_{m,n\in\ZZ}\{r\in U(m,n) \mid rU\cap U(t-d+2r_2,1)\subset K\} \]
and we want to show that $U(t-d+2r_2,1)\subset L$ (or equivalently, $L=U$). This will be done by proving first $z_1\in L$ and then $z_3\in L$.
Since $K\subset L$, we have
\[ 0 \equiv g_{11}z_1 + g_{33}z_3 = g_{11}z_1 + \mu^{r_3-r_2}g_{23}z_3 \equiv g_{11}z_1 \pmod L \;. \]
By Proposition~\ref{prop:lefschetz}, the complete intersection $f_1,f_2,g_{11}$ in $\CC[x_0,x_1,x_2]$ possesses the strong Lefschetz property.
We may thus assume that $\nu$ is a strong Lefschetz element for $f_1,f_2,g_{11}$. Lemma~\ref{lemma:gen4} then implies
\[ z_1U(m,0) \subset f_1z_1U + f_2z_1U + g_{11}z_1U + g_{12}z_1U \subset L \]
for all $m\ge\frac12(d_1+d_2+t-d+2r_1+t-d+r_1+r_2-3)$.
In order to show $z_1\in L$, we thus need to check that
\[ 2(t-d+2r_2+d_1) \ge d_1+d_2+t-d+2r_1+t-d+r_1+r_2-3 \;. \]
This is equivalent to
\[ 4r_2+2d_1 \ge d_1+d_2+3r_1+r_2-3 \;, \]
which simplifies to
$r_2 \ge r_1-3$. The last inequality is obviously true.

Next we show $z_3\in L$.
If $d_0>0$, we have
\[ \det\hat A_1 = \det\begin{pmatrix}
d_0x_0^{d_0-1} & d_2x_0^{d_2-1} & d_3x_0^{d_3-1} \\[.5em]
0 & d_2x_1^{d_2-1} & 0 \\[.5em]
0 & 0 & d_3x_2^{d_3-1}
\end{pmatrix} = d_0d_2d_3x_0^{d_0-1}x_1^{d_2-1}x_2^{d_3-1} \;. \]
Together with $K\subset L$ and $z_1\in L$, this implies
\[ 0\equiv (d_0d_2d_3)^{-1}x_1x_2\left(\det\hat A_1\right)z_3 = x_0^{d_0-1}x_1^{d_2}x_2^{d_3}z_3 \equiv x_0^{d_0+d_2+d_3-1}z_3 \pmod L \;. \]
Similarly, for $d_0=0$ we have
\[ \det\hat A_1 = \det\begin{pmatrix}
d_2x_0^{d_2-1} & d_3x_0^{d_3-1} \\[.5em]
d_2x_1^{d_2-1} & 0
\end{pmatrix} = -d_2d_3x_0^{d_3-1}x_1^{d_2-1} \]
and thus
\[ 0\equiv (d_2d_3)^{-1}x_1\left(\det\hat A_1\right)z_3 = -x_0^{d_3-1}x_1^{d_2}z_3 \equiv x_0^{d_0+d_2+d_3-1}z_3 \pmod L \]
as well.
By Proposition~\ref{prop:lefschetz}, the complete intersection $x_0^{d_0+d_2+d_3-1},f_2,f_3$ has the strong Lefschetz property.
Hence, we may assume that $\mu$ is a strong Lefschetz element for $x_0^{d_0+d_2+d_3-1},f_2,f_3$. Lemma~\ref{lemma:gen4} implies
\[ z_3U(m,0) \subset x_0^{d_0+d_2+d_3-1}z_3U + f_2z_3U + f_3z_3U + g_{23}z_3U \subset L \]
for all $m\ge\frac12(d_0+d_2+d_3-1+d_2+d_3+t-d+r_2+r_3-3)$.
It thus remains to check
\[ 2(t-d+2r_2+d_3) \ge d_0+d_2+d_3-1+d_2+d_3+t-d+r_2+r_3-3 \]
or
\[ 2r_0+2r_1+6r_2+6r_3+8d-6 \ge 3r_0+r_1+6r_2+6r_3+8d-7 \;. \]
This reduces to $r_1\ge r_0-1$, which is clearly true.
This finishes the proof of \eqref{eq:claimU}.
\end{proof}

Now let us take a monomial $hy_1^2y_2^2\in S(t,4)$ where $h\in S(t+2r_1+2r_2,0)$. We have $hz_1\in U(t-d+2r_2,1)$ and thus
\begin{align*}
hz_1 &= h_1f_1z_1 + h_2f_2 + h_3f_3z_3 + h_4g_{12}z_1 + h_5g_{23}z_3 \\
&\phantom{{}={}} {} + h_6(g_{11}z_1+g_{33}z_3) + h_7\left(\left(\det\hat A_3\right)z_1-\eps_{13}\left(\det\hat A_1\right)z_3\right)
\end{align*}
for certain homogeneous polynomials $h_1,\ldots,h_7\in U$. Substituting $z_j$ by $y_j^2$ for $j\in\{1,3\}$ and multiplying with $y_2^2$, we get
\begin{align*}
hy_1^2y_2^2 &= h_1f_1y_1^2y_2^2 + \tilde h_2f_2y_2^2 + h_3f_3y_2^2y_3^2 + h_4g_{12}y_1^2y_2^2 + h_5g_{23}y_2^2y_3^2 \\
&\phantom{{}={}} {} + h_6\left(g_{11}y_1^2y_2^2 + g_{33}y_2^2y_3^2\right) + h_7\left(\left(\det\hat A_3\right)y_1^2-\eps_{13}\left(\det\hat A_1\right)y_3^2\right)y_2^2 \\
&\equiv h_1y_1y_2^2\cdot0 + \tilde h_2y_2\cdot0 + h_3y_2^2y_3\cdot0 + h_4\cdot0 + h_5\cdot0 + h_6\cdot0 + h_7y_2^2\cdot0 \\
&\equiv0 \pmod J
\end{align*}
where we used the congruences \eqref{eq:fy}, \eqref{eq:gs2t2}, \eqref{eq:det}, \eqref{eq:det0}, and \eqref{eq:gs2t2+gt2u2},
and where $\tilde h_2$ denotes the result of the substitution inside $h_2$. This concludes the proof of $y_1^2y_2^2\in J$.

At this point, we are ready to handle the general case of $\{\sigma(0),\sigma(1)\}$. For this, we show the following claim:
\begin{claim}
If $\tau$ is a permutation of $\{0,1,2,3\}$ such that $\tau(3)<\tau(2)$,
then any multiple of $y_{\tau(0)}^2y_{\tau(1)}^2$ in $S(t,4)$ can be replaced modulo~$J$ by a multiple of $y_{\tau(0)}^2y_{\tau(2)}^2$ in $S(t,4)$.
\end{claim}
\begin{proof}[Proof of the claim]
In view of \eqref{eq:fy}, \eqref{eq:gs2t2}, \eqref{eq:det}, and \eqref{eq:det0}, it suffices to show that
\[ S(t+2r_{\tau(0)}+2r_{\tau(1)},0) \subset f_{\tau(0)}S + f_{\tau(1)}S + g_{\tau(0)\tau(1)}S + \left(\det\hat A_{\tau(2)}\right)S \;. \]
This will follow from Lemma~\ref{lemma:open} once we provide a special choice for the general polynomials
$f_{\tau(0)},f_{\tau(1)},f_{\tau(3)},g_{\tau(0)\tau(1)}$ satisfying this property.
Let $a=d_{\tau(0)}$, $b=d_{\tau(1)}$, and $c=d_{\tau(3)}$.
We may assume $a,b>0$ because otherwise we would already have $y_{\tau(0)}^2y_{\tau(1)}^2\equiv0\pmod J$ by \eqref{eq:fy}.
We take
\[ f_{\tau(0)} = x_0^a+x_1^a \;,\quad f_{\tau(1)} = x_0^b+x_2^b \;,\quad f_{\tau(3)} = x_0^c \;. \]
If also $c>0$, we have
\[ \det\hat A_{\tau(2)}=\pm\det\begin{pmatrix} ax_0^{a-1} & bx_0^{b-1} & cx_0^{c-1} \\ ax_1^{a-1} & 0 & 0 \\ 0 & bx_2^{b-1} & 0 \end{pmatrix}
= \pm abcx_0^{c-1}x_1^{a-1}x_2^{b-1} \;. \]
Therefore, we get
\[ x_0^{a+b+c-1} \in f_{\tau(0)}S + f_{\tau(1)}S + \left(\det\hat A_{\tau(2)}\right)S \;. \]
If $c=0$, it follows that $d_0=0$. Since $a,b>0$ and $\tau(3)<\tau(2)$, only $\tau(3)=0$ is possible. Then we have
\[ \det\hat A_{\tau(2)}=\pm\det\begin{pmatrix} ax_0^{a-1} & bx_0^{b-1} \\ ax_1^{a-1} & 0 \end{pmatrix} = \mp abx_0^{b-1}x_1^{a-1} \]
und thus again
\[ x_0^{a+b+c-1}=x_0^{a+b-1} \in f_{\tau(0)}S + f_{\tau(1)}S + \left(\det\hat A_{\tau(2)}\right)S \;. \]
In either case, the complete intersection $x_0^{a+b+c-1},f_{\tau(0)},f_{\tau(1)}$ has the strong Lefschetz property by Proposition~\ref{prop:lefschetz},
so we may pick for $g_{\tau(0)\tau(1)}$ an adequate power of a strong Lefschetz element and obtain via Lemma~\ref{lemma:gen4}
\[ S(m,0) \subset f_{\tau(0)}S + f_{\tau(1)}S + g_{\tau(0)\tau(1)}S + \left(\det\hat A_{\tau(2)}\right)S \]
for all $m\ge\frac12(a+b+c-1+a+b+t-d+r_{\tau(0)}+r_{\tau(1)}-3)$. Therefore, it remains to prove that
\[ 2(t+2r_{\tau(0)}+2r_{\tau(1)}) \ge a+b+c-1+a+b+t-d+r_{\tau(0)}+r_{\tau(1)}-3 \;. \]
This simplifies to
\[ 6r_{\tau(0)}+6r_{\tau(1)}+2r_{\tau(2)}+2r_{\tau(3)}+8d-6 \ge 6r_{\tau(0)}+6r_{\tau(1)}+r_{\tau(2)}+3r_{\tau(3)}+8d-7 \]
or just $r_{\tau(2)} \ge r_{\tau(3)}-1$, which holds because $\tau(3)<\tau(2)$.
\end{proof}
With this result at hand, we proceed as follows: We start with a monomial of degree~$(t,4)$ divisible by $y_{\sigma(0)}^2y_{\sigma(1)}^2$
and repeatedly apply transitions of the form \[ y_{\tau(0)}^2y_{\tau(1)}^2\leadsto y_{\tau(0)}^2y_{\tau(2)}^2 \] with $\tau(3)<\tau(2)$
for a suitable permutation $\tau$ until we arrive at a polynomial divisible by $y_1^2y_2^2$, for which we have already shown that it
vanishes modulo~$J$. The fact that such a sequence of transitions always exists can be most easily seen from the following diagram:
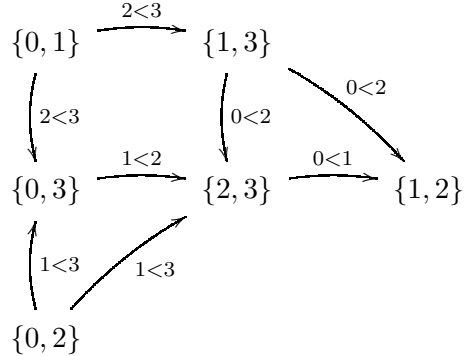
\[ \xymatrix@=3em@M+=6pt{
\{0,1\} \ar@/^/[r]^{2<3}\ar@/_/[d]^{2<3} & \{1,3\} \ar@/_/[d]^{0<2}\ar@/^/[rd]^{0<2} \\
\{0,3\} \ar@/^/[r]^{1<2} & \{2,3\} \ar@/^/[r]^{0<1} & \{1,2\} \\
\{0,2\} \ar@/^/[u]_{1<3}\ar@/^/[ru]_{1<3}
} \]
The arrows are labeled with the inequalities $\tau(3)<\tau(2)$ which hold for the employed permutations~$\tau$.
For every possible subset $\{\sigma(0),\sigma(1)\}\subset\{0,1,2,3\}$, there exists at least one directed path ending in $\{1,2\}$.
This completes the proof of step~\ref{step:s2t2}.
\end{proof}

\subsection{Fourth step}\label{step:s4}

We have $y_j^4\in J$ for all $j\in\{0,1,2,3\}$.
\begin{proof}
Let us take a monomial $hy_j^4$ where $h\in S(t+4r_j,0)$. If $d_j=0$, we are done by \eqref{eq:fy}.
Otherwise, multiplying \eqref{eq:dfy2} with $y_j^2$ and using step~\ref{step:s2t2} produces
\[ \frac{\partial f_j}{\partial x_k}y_j^4\equiv0\pmod J\;,\quad k\in\{0,1,2\}\;. \]
First suppose $j<3$. By Lemmas \ref{lemma:open} and \ref{lemma:gen3}, we have
\[ S(3d_j-5,0) \subset \frac{\partial f_j}{\partial x_0}S+\frac{\partial f_j}{\partial x_1}S+\frac{\partial f_j}{\partial x_2}S \]
since the partial derivatives of $f_j=x_0^{d_j}+x_1^{d_j}+x_2^{d_j}$ form a complete intersection.
Therefore, it remains to show that $t+4r_j \ge 3d_j-5$.
This is equivalent to
\[ r_0+r_1+r_2+r_3+4r_j+4d-3 \ge 6r_j+3d-5 \;, \]
which in turn is equivalent to
\[ r_0+r_1+r_2+r_3+d+2 \ge 2r_j \;. \]
The last inequality is true because $j\le2$ implies $r_2+r_3\ge r_j+r_j$.

Now let $j=3$. If we multiply \eqref{eq:g} with $y_3^2$ and use all previous steps, we obtain
\[ g_{33}y_3^4 \equiv 0 \pmod J \;. \]
We claim that
\[ S(t+4r_3,0) \subset \frac{\partial f_3}{\partial x_0}S+\frac{\partial f_3}{\partial x_1}S+\frac{\partial f_3}{\partial x_2}S+g_{33}S \;. \]
By Lemma~\ref{lemma:open}, it is enough to give one working example for $f_3$ and $g_{33}$.
If we take again $f_3=x_0^{d_3}+x_1^{d_3}+x_2^{d_3}$, the complete intersection given by the partial derivatives of $f_3$ has the
strong Lefschetz property by Proposition~\ref{prop:lefschetz}, so we may choose for $g_{33}$ a power of a strong Lefschetz element and obtain
via Lemma~\ref{lemma:gen4} that
\[ S(m,0) \subset \frac{\partial f_3}{\partial x_0}S+\frac{\partial f_3}{\partial x_1}S+\frac{\partial f_3}{\partial x_2}S+g_{33}S \]
for all $m\ge\frac12(3d_3-3+t-d+2r_3-3)$. Therefore, we are finished if
\[ 2(t+4r_3) \ge 3d_3-3+t-d+2r_3-3 \;. \]
This simplifies to
\[ 2r_0+2r_1+2r_2+10r_3+8d-6 \ge r_0+r_1+r_2+9r_3+6d-9 \;, \]
or equivalently,
\[ r_0+r_1+r_2+r_3+2d+3 \ge 0 \;. \]
The last statement is clearly true.
\end{proof}
Since every monomial in $S(t,4)$ is divisible by an element handled in one of the four steps above, we obtain $S(t,4)\subset J$ as desired.
This finally ends the proof of Proposition~\ref{prop:main}.
\begin{remark}
It was crucial in the choice of $g$ to leave out the terms $g_{00}$ and $g_{22}$,
i.\,e.\ the ones belonging to the smallest and second-largest values among the degrees $d_0,d_1,d_2,d_3$.
With any other two indices, the above proof would not work. Furthermore, if we would also set $g_{33}=0$,
the proof of step~\ref{step:s2t2} would be much simpler, but then step~\ref{step:s4} would work out only if $d_3\le d_0+d_1+d_2+4$.
And if we would instead set $g_{11}=0$, step~\ref{step:s4} could be left untouched, but step~\ref{step:s2t2}, though it would be simpler,
would turn out right only if $d_3\le d_2+6$. It is also worth to mention that the properties of $J$ we are proving in each of the four steps are in general \emph{not} open on the polynomials $f_j$ and $g_{ij}$,
thus an argument where one specializes to $g_{33}=0$ in one step but not in another one does not succeed.
\end{remark}

\bibliographystyle{myamsalpha}
\bibliography{biblio}
\end{document}